\documentclass[a4paper,11pt]{article}

\usepackage[utf8]{inputenc}
\usepackage[T1]{fontenc}
\usepackage{fullpage}
\usepackage{xspace}

\usepackage[english]{babel}
\usepackage{amsmath}
\usepackage{amsthm}
\usepackage{amssymb}
\usepackage{amsopn}
\usepackage{todonotes}

\theoremstyle{plain}
\newtheorem{theorem}{Theorem}

\newtheorem{proposition}[theorem]{Proposition}
\newtheorem{corollary}[theorem]{Corollary}
\newtheorem{lemma}[theorem]{Lemma}

\newtheorem{open.problem}{Open Problem}

\newtheorem*{theorem*}{Theorem}
\newtheorem*{corollary*}{Corollary}
\newtheorem*{proposition*}{Proposition*}
\newtheorem*{lemma*}{Lemma}
\newtheorem*{fact*}{Fact}
\newtheorem*{open.problem*}{Open Problem}
\newtheorem*{example*}{Example}
\newtheorem*{exercise*}{Exercise}

\theoremstyle{remark}
\newtheorem{claim}{Claim}
\newtheorem{remark}[theorem]{Remark}
\newtheorem*{claim*}{Claim}
\newtheorem*{remark*}{Remark}
\newtheorem{question}{Question}

\theoremstyle{definition}
\newtheorem{definition}{Definition}[section]

\DeclareMathOperator{\FS}{FS}

\begin{document}


\title{New bounds on the strength of some restrictions of Hindman's Theorem
\thanks{Part of this work was done while the first author was visiting the Institute for Mathematical Sciences, National University of Singapore in 2016. The visit was supported by the Institute. The second author was partially supported by Polish National Science Centre grant no. 2013/09/B/ST1/04390. The fourth author was partially supported by University Cardinal Stefan Wyszy\'nski in Warsaw grant UmoPBM-26/16. Some of the results have been presented at the conference Computability in Europe 2017 and appeared in an extended abstract in the proceedings of that conference.}}

\author{Lorenzo Carlucci \texttt{carlucci@di.uniroma1.it}
\and Leszek Aleksander Ko{\l}odziejczyk  \texttt{lak@mimuw.edu.pl}
\and Francesco Lepore  \texttt{leporefc@gmail.com}
\and Konrad Zdanowski \texttt{k.zdanowski@uksw.edu.pl}}


\newcommand\Nat{\mathbf{N}}
\newcommand\RCA{\mathsf{RCA}}
\newcommand\ACA{\mathsf{ACA}}
\newcommand\WKL{\mathsf{WKL}}
\newcommand\Ind{\mathsf{Ind}}

\newcommand\ATR{\mathsf{ATR}}
\newcommand\HT{\mathsf{HT}}
\newcommand\IPT{\mathsf{IPT}}
\newcommand\SRT{\mathsf{SRT}}
\newcommand\RT{\mathsf{RT}}
\newcommand\TT{\mathsf{TT}}

\newcommand\FUT{\mathsf{FUT}}
\newcommand\FU{\mathrm{FU}}

\newcommand\Pf{\mathcal{P}_{\mathrm{fin}}}

\newcommand\psRT{\mathsf{psRT}}
\newcommand\RWKL{\mathsf{RWKL}}
\newcommand\AHT{\mathsf{AHT}}
\newcommand\PHT{\mathsf{PHT}}
\newcommand\IPHT{\mathsf{IPHT}}
\newcommand\ADS{\mathsf{ADS}}
\newcommand\I{\mathsf{I}}
\newcommand\D{\mathsf{D}}

\newcommand{\rg}{\mathrm{rg}}

\newcommand{\imp}{\Rightarrow}
\newcommand{\revimp}{\Leftarrow}

\newcommand{\goed}[1]{ \ulcorner #1 \urcorner }
\newcommand{\num}[1]{\underline{#1}}
\newcommand{\tuple}[1]{\langle #1 \rangle}

\newcommand{\bool}{\mbox{\rm{bool}}}
\newcommand{\lin}{\mathrm{lin}}
\newcommand{\repr}{\mathrm{repr}}
\newcommand{\rem}{\mathrm{rem}}
\newcommand{\card}{\mathrm{card}}

\newcommand{\supexp}{\mathrm{supexp}}
\newcommand{\Log}{\mathrm{Log}}

\newcommand{\bbn}{\mathbb{N}}
\newcommand{\bbr}{\mathbb{R}}
\newcommand{\bbq}{\mathbb{Q}}
\newcommand{\bbz}{\mathbb{Z}}
\newcommand{\bba}{\mathbb{A}}

\newcommand{\ea}{\mathrm{EA}}
\newcommand{\pa}{\mathrm{PA}}
\newcommand{\pamin}{\mathrm{PA^-}}
\newcommand{\pr}{\mathrm{Pr}}

\newcommand{\rca}{\mathsf{RCA}_0}
\newcommand{\rcastar}{\mathsf{RCA}^*_0}
\newcommand{\wkl}{\mathsf{WKL}_0}
\newcommand{\wklstar}{\mathsf{WKL}^*_0}
\newcommand{\aca}{\mathsf{ACA}_0}
\newcommand{\wklrm}{\mathrm{WKL}}

\newcommand{\PHP}{\mathrm{PHP}}

\newcommand{\ido}{\mathrm{I}\Delta_0}
\newcommand{\bsi}{\mathrm{B}\Sigma_1}
\newcommand{\bs}[1]{\mathrm{B}\Sigma_{#1}}
\newcommand{\is}[1]{\mathrm{I}\Sigma_{#1}}

\newcommand{\n}{{\mathcal N}}
\newcommand{\m}{{\mathcal M}}
\newcommand{\X}{{\mathcal X}}
\newcommand{\Y}{{\mathcal Y}}

\newcommand{\es}[2]{\exists #1 \! < \! #2 \,}   
\newcommand{\as}[2]{\forall #1 \! < \! #2 \,}
\newcommand{\eb}[2]{\exists #1 \! \le \! #2 \,}
\newcommand{\ab}[2]{\forall #1 \! \le \! #2 \,}

\renewcommand{\restriction}{\mathord{\upharpoonright}}

\newcommand\apart[1]{#1\mbox{-}\mathsf{apart}}

\newcommand{\kznote}[1]{\footnote{kz: #1}}
\newcommand{\lknote}[1]{\footnote{lak: #1}}



\maketitle

\begin{abstract}
The relations between (restrictions of) Hindman's Finite Sums Theorem and (variants of) Ramsey's Theorem 
give rise to long-standing open problems in combinatorics, computability theory and proof theory.
We present some results motivated by these open problems. 
In particular we investigate the restriction of the Finite Sums Theorem
to sums of one or two elements, which is the subject of a long-standing open question 
by Hindman, Leader and Strauss. We show that this restriction 
has the same proof-theoretical and computability-theoretic lower bound that is known to hold for the 
full version of the Finite Sums Theorem. In terms of reverse mathematics, it implies $\ACA_0$. 
Also, we show that Hindman's Theorem restricted to sums of exactly $n\geq 3$ elements, is equivalent to 
$\ACA_0$, provided a certain sparsity condition is imposed on the solution set. 
The same results apply to bounded versions of the Finite Union Theorem, in which such a sparsity 
condition is built-in.
Further we show that the Finite Sums Theorem for sums of at most two elements is tightly 
connected to the Increasing Polarized Ramsey's Theorem for pairs introduced by Dzhafarov and Hirst.
The latter reduces to the former in a strong technical sense known as strong computable reducibility, 
which essentially means that there is a natural combinatorial reduction proof of one principle to the other. 
\end{abstract}



\section{Introduction and motivation}\label{sec:intro}

The Finite Sums Theorem by Neil Hindman~\cite{Hin:74} says that whenever the positive integers 
are coloured in finitely many colours there exists an infinite set of positive integers such that all the finite 
non-empty sums of distinct numbers from the set have the same colour. 
We denote this theorem by $\HT$ and use $\HT_k$ to stand for its restriction to $k$-colourings.
Writing $\FS(X)$ 
for the set of non-empty finite sums of distinct elements of the set $X$, the conclusion 
of Hindman's Theorem is that there exists an infinite $X\subseteq \Nat$ ($\Nat$ denotes the set of positive
integers throughout the paper) such that $\FS(X)$ is monochromatic.

There are some interesting long-standing open problems related to $\HT$ 
at the crossroads of combinatorics, proof theory and computability theory. 
The following question was asked by Hindman, Leader and Strauss in~\cite{Hin-Lea-Str:03}, 
and has been open since. 
\begin{quote}
{\em Question 12}. Is there a proof that whenever $\Nat$ is finitely coloured there is a sequence
$x_1, x_2, \dots$ such that all $x_i$ and all $x_i + x_j$ ($i\neq j$) 
have the same colour, that does not also prove the Finite Sums Theorem?
\end{quote}
It is very natural to recast the above question in the context of reverse mathematics, 
which is a framework for rigorously comparing the relative strength of theorems 
from all areas of mathematics over a fixed base theory (see~\cite{Sim:SOSOA,Hir:STT:14} for excellent introductions
to the topic). Traditionally such a base theory is the formal axiomatic
system $\RCA_0$ ($\RCA_0$ is mmenomic for 
Recursive Comprehension Axiom) capturing the intuitive idea of \emph{computable 
mathematics}.\footnote{Loosely speaking this means mathematics done using only constructions which could
be performed by a computer program, regardless of time and space constraints.}  
Denoting by $\HT^{\leq n}$ the restriction of $\HT$ to (non-empty) sums of at most $n$ distinct
elements, and by $\HT^{\leq n}_k$ the further restriction to $k$-colourings,
a good formal rendering of Question 12 reads as follows: 
Is $\HT^{\leq 2}$ enough to prove $\HT$ over $\RCA_0$?

Pinning down the exact strength of Hindman's Theorem 
is by itself one of the major open problems in reverse mathematics (see~\cite[Question 9]{Mon:11:open}). 
The seminal results of Blass, Hirst and Simpson in the late eighties leave indeed a huge gap between 
lower and upper bound. 
In terms of reverse mathematics these results place Hindman's Theorem not lower than the system $\ACA_0$ (Arithmetical Comprehension Axiom\footnote{This axiom
guarantees the existence of all sets of natural numbers that can be defined without
quantifying over infinite objects but with otherwise no bound on the number of alternations of quantifiers in the defining formula. 
It is equivalent to asserting that the Turing Jump of any set exists.}) 
and not higher than the much stronger system $\ACA_0^+$.\footnote{The difference is the same between being able to decide the Halting Set and being able to decide any arithmetical truth about the natural numbers. The system $\ACA_0^+$ extends
$\ACA_0$ by the axiom stating that the $\omega$-th Turing jump is always defined. Recall that the $\omega$-th Turing Jump of the empty set is the degree of unsolvability of arithmetical truth. 
} 
Note that $\ACA_0$ is known to be equivalent to $\RT^3_2$ (Ramsey's Theorem 
for 2-colorings of triples) by seminal work of Jockusch and of Simpson (\cite[Theorem III.7.6]{Sim:SOSOA} or \cite[Chapter 6]{Hir:STT:14}), so we have that $\HT$ implies $\RT^3_2$ over $\RCA_0$. 
On the other hand $\ACA_0^+$ was only recently given a Ramsey-theoretic characterization
in work of the first and fourth author, who showed \cite{Car-Zda:14} that the system $\ACA_0^+$ is equivalent
to a Ramsey-theoretic theorem due to Pudl\'ak and R\"odl \cite{Pud-Rod:82} and Farmaki and Negrepontis \cite{Far-Neg:08}, 
which we denote by $\RT^{!\omega}_2$ (see 
Definition \ref{defi:rtomega}). 
This theorem extends Ramsey's Theorem to colourings of objects of variable dimension, in particular to 
so-called {\em exactly large sets} of integers, where a set is exactly large in case its cardinality is greater by one than
its minimum element.
The following inequalities summarize the situation with respect to implications over the base theory $\RCA_0$:
\[\RT^{!\omega}_2 \to \HT \to \RT^3_2,\]
where at least one of the two implications does not reverse, because 
it is known that $\RT^3_2 \nrightarrow \RT^{!\omega}_2$ (in fact, 
$\forall n \forall k \RT^n_k \nrightarrow \RT^{!\omega}_2$).

In terms of computability theory, the Blass-Hirst-Simpson's bounds on $\HT$ can be expressed as follows. 
On the one hand, there exists a computable coloring $c:\Nat\to 2$ such that any solution
to Hindman's Theorem for the coloring\footnote{By a ``solution to Hindman's Theorem for the coloring $c$'' we mean an infinite set $H$ such that all finite non-empty sums of elements from $H$ have the same $c$-color.} $c$ computes $\emptyset'$, the first Turing Jump of the computable sets. On the other hand, for every computable coloring $c:\Nat\to 2$ there exists a solution set
computable from $\emptyset^{(\omega+1)}$, the $(\omega+1)$-th Turing Jump of the computable sets.

In \cite{Bla:05} Blass advocated the study of restrictions of Hindman's Theorem to sums of bounded length (i.e., 
number of terms), conjecturing that the strength of $\HT$ grows with the length of the sums for which monochromaticity is required. 
Only recently Dzhafarov, Jockusch, Solomon and Westrick~\cite{DJSW:16} proved that 
the restriction of $\HT$ to sums of at most 3 terms from the solution set, $\HT^{\leq 3}$, already implies
$\ACA_0$, thus realizing the only known lower bound for $\HT$ (in particular, $\HT^{\leq 3}_3$ suffices).

One of our main results is that the same lower bound 
already holds for the restriction to sums 
of at most 2 elements, $\HT^{\leq 2}$, i.e., the 
restriction of $\HT$ considered in \cite[Question 12]{Hin-Lea-Str:03}.
This means that the known upper and lower bounds for $\HT$ and $\HT^{\leq 2}$ are now the same,
which might be read as indicating that the restriction of $\HT$ to sums of at most two terms 
is close in strength to the full theorem. 

On the other hand, we prove that the same lower bound holds for a number of
restricted forms of $\HT$ for which a matching upper bound can also be proved. 
The first examples of principles with this property, at the level of $\ACA_0$, 
were found in \cite{Car:16:wys} and therein called ``weak yet strong'' principles. 
We improve and expand on \cite{Car:16:wys} by showing, 
for example, that Hindman's Theorem for sums of exactly $n$ elements --- denoted $\HT^{=n}_k$, for $k$-colourings 
--- is equivalent to $\ACA_0$, provided that $n\geq 3$ and a certain sparsity condition is imposed on the solution set. 
Such a condition, which we call the apartness condition,  
is crucial yet not given a name in earlier work \cite{Hin:72,Bla-Hir-Sim:87,DJSW:16}. In our setting it means that the sets of exponents
in some fixed base of the elements of the homogeneous set do not intertwine. An analogous condition is built-in in 
the formulation of Hindman's Theorem in terms of finite unions (the Finite Unions Theorem), 
and called the unmeshedness condition (\cite{Bla:05}) or 
the block sequence condition (\cite{Arg-Tod:05}). We will observe that bounded versions of the Finite 
Unions Theorem are equivalent to bounded versions of the Finite Sums Theorem with the apartness condition. 

Note that, in contrast to $\HT^{\leq n}_k$, the exact versions of Hindman's Theorem $\HT^{=n}_k$
are easily seen to follow from $\RT^n_k$: 
given a colouring $f:\Nat\to k$, let $c:[\Nat]^n \to k$
be defined by setting $c(a_1,\dots,a_n) = f(a_1+\dots+a_n)$. A solution to 
$\RT^n_k$ for the instance $c$ (i.e., an infinite monochromatic set $X$) is a solution to $\HT^{=n}_k$ for 
instance $f$ (i.e., $\FS^{= n}(X)$ is monochromatic, where we denote by $\FS^{=n}(X)$ the set
of sums of exactly $n$ many distinct elements of $X$).
We will prove, for example, that $\RT^3_2$ already follows from (and is actually equivalent to)
$\HT^{=3}_2$ with the apartness condition imposed on the solution set. 

The argument just given is an example of a particularly simple and natural combinatorial reduction of the principle 
$\HT^{=n}_k$ to $\RT^n_k$: Starting from an instance $f$ of $\HT^{= n}_k$ we defined an instance $c$ of $\RT^{n}_{k}$. 
From a solution $X$ to $c$ we recovered a solution $X'$ to the 
original instance of $\HT^{=n}_k$ (in that case $X'$ equals $X$). Proofs of this kind are abundant in combinatorics.
Furthermore observe that in the above example $c$ is easily seen to be computable relative to\footnote{Informally, this
means that the transition from $f$ to $c$ can can be done by a computer assuming it has access to values of $f$.} $f$
and similarly $X'$ is computable relative to $X$ (this is obvious since $X=X'$ in the example at hand). 
Such a proof that $\RT^n_k$ follows from $\HT^{=n}_k$ is an instance of what is known in the literature as 
a {\em strong computable reduction}.
This notion, first defined in \cite{Dza:15}, has quickly become central in the computable and 
reverse mathematics literature (see, e.g., \cite{Dza:16} and references therein). We use the notation 
$\mathsf{Q}\leq_{\mathrm{sc}}\mathsf{P}$ to indicate that a Ramsey-type theorem $\mathsf{Q}$ is
reducible to another Ramsey-type theorem $\mathsf{P}$ by a strong computable reduction. 
Not all proofs of an implication over $\RCA_0$ have the form of a strong computable reduction. 
For example, it has been recently proved \cite{Pat:16} that there is no strong computable
reduction from $\RT^n_3$ to $\RT^n_2$, despite the fact that a straightforward combinatorial
argument exists and that the two theorems are equivalent over $\RCA_0$.
In the present paper, however, we only deal with positive results.
For example we prove that 
an interesting restriction of Ramsey's Theorem for pairs (the Increasing Polarized Ramsey's Theorem 
of Dzhafarov and Hirst's \cite{Dza-Hir:11}, denoted $\IPT^2_2$) is strongly computably reducible 
to $\HT^{\leq 2}_4$ (in fact to $\HT^{=2}_2$ with the apartness condition imposed on the solution set).

The paper is organized as follows. In Section~\ref{sec:apart} we define the apartness condition and prove a 
simple lemma about it, and discuss the equivalence of the bounded versions of the Finite Unions Theorem with 
bounded versions of the Finite Sums Theorem with apartness. 
In Section~\ref{sec:aca} we prove $\ACA_0$ lower bounds for restrictions of Hindman's Theorem,
including our main result that $\HT^{\le 2}$ implies $\ACA_0$ over $\RCA_0$.
In Section~\ref{sec:ipt} we deal with reductions between Hindman's Theorem and the
Increasing Polarized Ramsey's Theorem.
In Section~\ref{sec:misc} we present a number of other results that can be 
obtained by the arguments of the previous sections. 
In Section~\ref{sec:conclusion} we summarize our results and 
discuss some open problems.

\section{Hindman's Theorem, apartness, and finite unions}\label{sec:apart}

We define two natural types of restrictions of Hindman's Theorem based on bounding the 
length of sums for which homogeneity is guaranteed. 

\begin{definition}[Hindman's Theorem with bounded-length sums]
Fix $n,k\geq 1$. 
\begin{enumerate}
\item $\HT^{\leq n}_k$ is the following principle: For every coloring $f~:~\Nat\to k$ there exists an infinite set 
$H\subseteq\Nat$ such that $FS^{\leq n}(H)$ is monochromatic for $f$. 
\item $\HT^{=n}_k$) 
is the following principle: For every coloring $f~:~\Nat\to k$ there exists an infinite set 
$H\subseteq\Nat$ such that $FS^{=n}(H)$ is monochromatic for $f$. 
\end{enumerate}
\end{definition}
The principles $\HT^{\leq n}_k$ were discussed in~\cite{Bla:05} (albeit phrased in terms of finite
unions instead of sums) and first studied from the perspective of Computable and Reverse 
Mathematics in~\cite{DJSW:16}, where the principles $\HT^{=n}_k$ were also defined.

As indicated above, some of our results highlight the crucial role of a property of the solution set 
-- which we call the apartness condition -- that is central in Hindman's original
proof and in the proofs of the lower bounds in~\cite{Bla-Hir-Sim:87,DJSW:16,Car:16:wys}. 


We use the following notation: Fix a base 
$t\geq 2$. For $n\in \Nat$
we denote by $\lambda_t(n)$ 
the least exponent of $n$ written in base $t$, by $\mu_t(n)$ the greatest exponent of $n$ written in base $t$.
We will drop the subscript when clear from context.

\begin{definition}[Apartness Condition]
Fix $t\geq 2$. We say that a set $X\subseteq\Nat$ satisfies the $t$-apartness condition (or {\em is $t$-apart}) if for all $x,x'\in X$, 
if $x<x'$ then $\mu_t(x)<\lambda_t(x')$. 
\end{definition}
Note that $t$-apartness is inherited by subsets.

For a Hindman-type principle $\mathsf{P}$,
let ``{\em $\mathsf{P}$ with $t$-apartness}'' denote the corresponding version in which 
the solution set is required to satisfy the $t$-apartness condition. 
As will be observed below, it is significantly easier to prove lower bounds on $\mathsf{P}$ with 
$t$-apartness than on $\mathsf{P}$ in all the cases we consider. 
In Hindman's original paper it is shown \cite[Lemma 2.2]{Hin:74} how 2-apartness can be ensured 
by a simple counting argument (proved in \cite[Lemma 2.2]{Hin:72})
under the assumption that we have a solution to the Finite Sums Theorem, i.e., an infinite $H$ 
such that $\FS(H)$ is monochromatic. In our terminology, we have that, for each $k\in\Nat$, 
$\HT_k$ is equivalent to $\HT_k$ with 2-apartness. Note that the counting argument used by Hindman 
\cite[Lemma 2.2]{Hin:72} 
requires very elementary arithmetic assumptions, and that the set satisfying $t$-apartness is obtained
from a general solution to $\HT$ by an algorithmic thinning out procedure
(as 
observed already in \cite{Bla-Hir-Sim:87}). In other words, $\HT$ and $\HT$ with $t$-apartness are
equivalent over $\RCA_0$.
\begin{proposition}[Implicit in \cite{Hin:72}]
For each positive integers $t$ and $k$, 
$\HT_k$ and $\HT_k$ with $t$-apartness are equivalent over $\RCA_0$. The equivalence is witnessed by strong computable reductions. 
\end{proposition}
Note that, to show the implication from $\HT_k$ to $\HT_k$ with $t$-apartness 
it is crucial that we start with a homogeneous set $H$ such that all finite sums of 
distinct elements from $H$ have the same colour. Putting a bound on the length of the sums would disrupt 
the argument. 
Thus, for bounded versions of $\HT$, the situation might be different. However, in typical situations, 
the choice of $t$ in $t$-apartness does not matter. We prove below 
that $\HT^{\leq n}_k$ with $t$-apartness  and $\HT^{=n}_k$ with $t$-apartness are 
robust concepts and that it is sufficient to consider the case of $t=2$. To show this in detail we make a detour through
another popular formulation of Hindman's Theorem in terms of colorings of finite subsets of the 
natural numbers (see, e.g., \cite{Ber:10}). This version is called the Finite Union Theorem. Let $\Pf(X)$ denote the set of finite subsets of $X$. Let $\Nat_0$ denote
$\Nat\cup\{0\}$. 
If $(X_i)_{i\in\Nat}$ is a sequence of finite subsets of $\Nat$, 
we denote by $\FU((X_i)_{i\in\Nat})$ the set of all finite unions of elements of $(X_i)_{i\in\Nat}$, i.e., 
$\FUT((X_i)_{i\in\Nat} = \{ \bigcup_{t \in F} X_t \,:\, F \mbox{ a non-empty finite subset of }\Nat_0\}$. 
\begin{definition}[Finite Unions Theorem]
$\FUT_k$: For every $f:\Pf(\Nat_0)\to k$ there exists an infinite sequence $(X_i)_{i\in\Nat}$ of finite subsets of $\Nat$ such that
if $i< j$ then $\max(X_i) < \min(X_j)$ and such that $\FU((X_i)_{i\in \Nat}$ is monochromatic. $\FUT$ denotes $\forall k \FUT_k$.
\end{definition}
A sequence $(X_i)_{i\in\Nat}$ of finite sets is called {\it unmeshed} or a {\it block sequence} if it satisfies
the condition that for each $i < j$ then $\max(X_i) < \min(X_j)$. This condition is obviously akin to apartness
and is part of the very statement of the Finite Unions Theorem. If this requirement is dropped, then the theorem becomes equivalent to 
the Infinite Pigeonhole Principle $\forall k \RT^1_k$ as proved by Hirst in \cite{Hir:12:HvsH}. 

The equivalence of $\HT$ with $\FUT$ is well-known (see, e.g., \cite{Ber:10}) and an inspection of the proof shows that
it is witnessed by strong computable reductions.  We below verify that the equivalence still holds between $\FUT^{\leq n}_k$ (resp. $\FUT^{=n}_k$) and $\HT^{\leq n}_k$ with $t$-apartness (resp. $\HT^{=n}_k$  with $t$-apartness), for any $t$, where $\FUT^{\leq n}_k$ and 
$\FUT^{=n}_k$ have the obvious meanings. 

This shows that the principles $\HT^{\leq n}_k$ with $2$-apartness are the natural bounded restrictions of $\HT$. 
Thus, we will only need to consider $2$-apartness in what follows, despite our use of $3$-apartness in 
Lemma \ref{lem:leqn_to_apart}.

\begin{proposition}\label{prop:ht_eq_fut}
For each $n,kt\geq 2$, $\HT^{\leq n}_{k}$ with $t$-apartness is equivalent to $\FUT^{\leq n}_k$ over $\RCA_0$.
Moreover, these principles are mutually strongly computably reducible.  The same equivalences hold for $\HT^{=n}_k$ with $t$-apartness
and $\FUT^{=n}_k$.
\end{proposition}
\begin{proof}
We give the proof for $\FUT^{\leq n}_k$ and $\HT^{\leq n}_k$ with $t$-apartness. For $\FUT^{=n}_k$ and 
$\HT^{=n}_k$ with $t$-apartness the argument is exactly analogous.

Let $c:\Pf(\Nat_0)\to k$. Define $d:\Nat\to k$ as follows: $d$ colors $m\in\Nat$ as $c$ colors the set of 
its base $t$ exponents. 
By $\HT^{\leq n}_{k}$ with $t$-apartness let $H=\{h_1 < h_2 < \dots\}$ be a $t$-apart 
infinite set such that $\FS^{\leq n}(H)$ is monochromatic for $d$. For each $i\in\Nat$ let
$S_i$ be the set of base $t$ exponents of $h_i$. Then $(S_i)_{i\in \Nat}$ is a block sequence
in $\Pf(\Nat)$ such that $c$ is constant on $\FU^{\leq n}((S_i)_{ i \in\Nat})$. 

Let $d:\Nat\to k$. Define $c:\Pf(\Nat)\to k$ as follows: $c$ colors $S\in\Pf(\Nat)$ as $d$ colors $t^{s_1}+\dots + t^{s_p}$
where $S=\{s_1 < \dots < s_p\}$. Let $d$ color the other elements of $\Nat$ arbitrarily. 
Let $(S_i)_{i\in\Nat}$ be a block sequence such that $\FU^{\leq n}((S_i)_{i\in\Nat})$
is monochromatic for $c$. Let $S_i = \{ s^i_1  < \dots < s^i_{p_i}\}$. 
Then $\{x_i \,;\, x_i = t^{s^i_1}+\dots+t^{s^i_{p_i}}, i \in\Nat\}$ is a $t$-apart solution 
to $\HT^{\leq n}_k$ for $d$. 
\end{proof}


\begin{corollary}\label{cor:apart}
Over $\RCA_0$, $\HT^{\leq n}_k$ with $t$-apartness (resp. $\HT^{=n}_k$ with $t$-apartness) is equivalent to $\HT^{\leq n}_k$ 
with $s$-apartness (resp. $\HT^{=n}_k$ with $s$-apartness), for any $t,s\geq 2$.
\end{corollary}

Henceforth we will use just apartness for $2$-apartness. Note that, in what follows, all the results for $\HT^{\leq n}_k$ with apartness
(resp. $\HT^{= n}_k$ with apartness) also hold in the case of $\FUT^{\leq n}_k$ (eq., for $\FUT^{=n}_k$). 

In some cases it is easy to show that the apartness condition can be enforced at no cost. 
For example the proof of
$\HT^{=n}_k$ from 
$\RT^n_k$
sketched above yields $t$-apartness for any $t>1$ simply by applying Ramsey's Theorem 
relative to an infinite $t$-apart set. 
In some other cases the apartness condition can be ensured at the cost of increasing
the number of colours. This is the case of $\HT^{\leq n}_k$, as illustrated by the next lemma. The idea of the proof is
from the first part of the proof of \cite[Theorem 3.1]{DJSW:16}, with some needed adjustments.

\begin{lemma}[$\RCA_0$]\label{lem:leqn_to_apart} 
For all $n \geq 2$, for all $k \geq 1$, $\HT^{\leq n}_{2k}$ implies 
$\HT^{\leq n}_k$ with apartness. Furthermore, the implication is established by a strong computable reduction.
\end{lemma}

\begin{proof}
We work in base 3 (this is without loss of generality by Corollary \ref{cor:apart}). 
Let $f:\Nat\to k$ be given. Let $i(n)$ denote the coefficient of the least term of $n$ written in base $3$. 
Define $g:\Nat\to 2k$ as follows.
\[
g(n):=
\begin{cases}
f(n) & \mbox{ if } i(n)=1,\\
k+f(n)& \mbox{ if } i(n)=2.\\
\end{cases}
\]
Let $H$ be an infinite set such that $\FS^{\leq n}(H)$  is homogeneous for $g$ of colour $\ell$.
For $h,h'\in \FS^{\leq n}(H)$ we have $i(h)=i(h')$.

We claim that for each $j\geq 0$
there is at most one $h\in H$ such that $\lambda(h)=j$. By way of contradiction
suppose otherwise, as witnessed by $h,h'\in H$. Then $i(h)=i(h')$ and $\lambda(h)=\lambda(h')$. Therefore
$i(h+h')= 3 - i(h) \neq i(h)$, but $h+h'\in \FS^{\leq n}(H)$. Contradiction. 

Using the claim, we can computably obtain a 3-apart infinite subset $H'$ of $H$.
\end{proof}

\section{Bounded Hindman vs.~Ramsey}\label{sec:aca}

In this section we first show that $\HT^{\leq 2}$ implies $\ACA_0$ (hence $\RT^3_2$) over $\RCA_0$. 
This improves on the main result of \cite{DJSW:16} that $\HT^{\leq 3}$ implies $\ACA_0$.
In particular we show that $\HT^{\leq 2}_4$ implies $\ACA_0$. In terms of finite unions our proof shows $\FUT^{\leq 2}_2$ implies $\ACA_0$.
This should also be compared with Corollary 2.3 and Corollary 3.4 of \cite{DJSW:16}, showing, respectively, that
$\HT^{\leq 2}_2$ implies the Stable Ramsey's Theorem $\SRT^2_2$ over the slightly stronger base theory $\RCA_0 + \mathsf{B}\Sigma^0_2$
or, equivalently, $\RCA_0 + \forall k \RT^1_k$).
Then we go on to prove that $\HT^{=3}_2$ with apartness is equivalent to 
$\ACA_0$. In terms of finite unions this shows that $\FUT^{=3}_2$ is equivalent to $\ACA_0$. 
Note that while $\HT^{=3}_2$ with apartness is easily reducible to $\RT^3_2$, 
it is unknown whether $\ACA_0$ (and thus $\RT^3_2$) implies $\HT^{\leq 2}_2$ over $\RCA_0$.

The lower bound proofs below are based on a significant simplification of the original argument of 
Blass, Hirst and Simpson \cite{Bla-Hir-Sim:87}.\footnote{
Blass, towards the end of \cite{Bla:05}, states without giving details that
inspection of the proof of the lower bound for $\HT$ in \cite{Bla-Hir-Sim:87} shows that 
this bound also holds for the restriction of the Finite Unions Theorem to unions of at most
two sets. While our Proposition \ref{prop:apht22_to_aca0} confirms this conclusion, we would like to stress that
from an inspection of the proof in \cite{Bla-Hir-Sim:87} one can glean that sums of 3 elements are sufficient. 
Indeed, while apparently only sums of 2 terms are used, in one crucial step one of the summands
is itself a sum of length 2. 
}

\subsection{Sums of at most two terms}

Let us recall that in $\RCA_0$ we have that for every $n$ there exists some $\ell$ such that 
for each $x < n$, $x \in \rg(f)$ if and only if $x \in \rg(f\restriction\ell)$.
This is a special case of a general principle known as strong $\Sigma^0_1$-collection (or 
strong $\Sigma^0_1$-bounding, see~\cite[Exercise II.3.14]{Sim:SOSOA},~\cite[Thm I.2.23 and Definition I.2.20]{Haj-Pud:book}).
This simple fact will be used in our lower bound arguments below. 

\begin{proposition}\label{prop:apht22_to_aca0}
$\HT^{\le 2}_2$ with apartness (eq. $\FUT^{\leq 2}_2$) implies $\ACA_0$ over $\RCA_0$.
\end{proposition}

\begin{proof}
Assume $\HT^{\le 2}_2$ with apartness and consider an injective function
$f \colon \Nat \to \Nat$. We have to prove that the range of $f$ exists.\footnote{This is well-known to be equivalent to proving $\ACA_0$, see \cite[Lemma III.1.3 and Theorem III.7.6]{Sim:SOSOA}.} 

For a number $n$, written as  $2^{n_0} + \dots + 2^{n_r}$ in base $2$ notation, we call $j \in \{0,\ldots,r\}$ 
\emph{important in $n$} if some value of $f \restriction [n_{j-1},n_j)$ is below $n_0$. 
Here $n_{-1} := 0$. The colouring $g \colon \Nat \to 2$ 
is defined by
\[g(n) := \card\{j: j\textrm{ is important in } n\} \bmod 2.\]
Note that $g$ is computable relative to $f$.
By $\HT^{\le 2}_2$ with apartness, there exists an infinite set $H\subseteq\Nat$ such that $H$ is apart and $\FS^{\le 2}(H)$ 
is monochromatic w.r.t.\ $g$. We claim that for each $n \in H$ and each $x < \lambda(n)$, $x \in \rg(f)$ if and only if $x \in \rg(f \restriction \mu(n))$. 
This will give us an algorithm for deciding whether any given $x$ is in the range of $f$: find the smallest $n \in H$ such that
$x < \lambda(n)$ and check whether $x$ is in $\rg(f \restriction \mu(n))$.

It remains to prove the claim. 
In order to do this, consider $n \in H$ and assume that there is some element below $n_0 = \lambda(n)$ in 
$\rg (f) \setminus \rg(f \restriction \mu(n))$. 

Let $\ell$ be such that for each $x < \lambda(n)$, $x \in \rg(f)$ if and only if $x \in \rg(f\restriction\ell)$.
By apartness, and the fact that $H$ is infinite, there is $m \in H$ with $\lambda(m) \ge \ell > \mu(n)$. 
Write $n+ m$ in base $2$ notation,
\[n+m =  2^{n_0} + \dots + 2^{n_r} + 2^{n_{r+1}}  +  \dots + 2^{n_s},\]
where $n_0 = \lambda(n) = \lambda(n+m)$, $n_r = \mu(n)$, and $n_{r+1} = \lambda(m)$.  Clearly,
$j \le s$ is important in $n+m$ if and only if either (i) $j \le r$ and $j$ is important in $n$ or (ii) $j = r+1$; hence, $g(n) \neq g(n+m)$.
This contradicts the assumption that $\FS^{\le 2}(H)$ 
is monochromatic, thus proving the claim.   
\end{proof}

\begin{theorem}\label{thm:ht24_to_aca0}
$\HT^{\leq 2}_4$ implies $\ACA_0$ over  $\RCA_0$.
\end{theorem}

\begin{proof}
By Proposition~\ref{prop:apht22_to_aca0}, Lemma~\ref{lem:leqn_to_apart} and Corollary \ref{cor:apart}.
\end{proof}

\subsubsection{Sums of exactly three terms, with apartness}

We next extend the argument in Proposition~\ref{prop:apht22_to_aca0} to show that 
$\HT^{=3}_2$ with apartness implies $\ACA_0$ (hence $\RT^3_2$) over $\RCA_0$. 
Since $\HT^{=3}_2$ with apartness is also easily deducible from $\RT^3_2$, we obtain an equivalence.
Note that no lower bounds on $\HT^{=3}_2$ {\em without} apartness are known.

\begin{theorem}\label{thm:hteq3_to_aca0}
$\HT^{= 3}_2$ with apartness (eq., $\FUT^{=3}_2$) is equivalent to $\RT^3_2$ over $\RCA_0$.
\end{theorem}
\begin{proof}
The upper bound, that is the implication from $\RT^3_2$
to $\HT^{= 3}_2$ with apartness, follows by applying the argument
proving $\HT^{=n}_k$ from $\RT^n_k$ sketched in Section \ref{sec:intro}.
Thus, it remains to prove the lower bound.

We argue in the base theory $\RCA_0$ assuming $\HT^{= 3}_2$ with apartness. 
Consider an injective function 
$f \colon \Nat \to \Nat$. We have to prove that the range of $f$ exists.
The relation $j$ {\em is important in} $n$ and the colouring $g:\Nat\to 2$
are defined as in the proof of Proposition~\ref{prop:apht22_to_aca0}.

By $\HT^{=3}_2$ with apartness, 
there exists an infinite set $H$ such that $H$ is apart and $\FS^{=3}(H)$ 
is monochromatic w.r.t.\ $g$. Let $r<2$ be the colour of $\FS^{=3}(H)$ under $g$. 
We describe a method for algorithmically deciding membership in the 
range of $f$ relative to the set $H$. 

\begin{claim}\label{lem:deciding}
For each $n,k\in H$. If $n<k$ and $g(n+k)=r$ then for each $x<\lambda(n)$,
\[
x\in \rg(f) \iff x\in \rg(f\restriction \mu(k)).
\]
\end{claim}
\noindent To prove Claim \ref{lem:deciding}, let $n,k\in H$ be such that $n<k$ and $g(n+k)=r$.
As in the proof of Proposition~\ref{prop:apht22_to_aca0}, let $\ell$ be such that 
for all $x<\lambda(n)$, 
\[
x\in \rg(f) \iff x\in \rg(f\restriction \ell).
\]
Then, take $m\in H$ such that $\lambda(m)>\ell$.
Now, if $x\in \rg(f) \setminus \rg(f\restriction \mu(k))$ for some $x<\lambda(n)$,
then the number of important digits in $n+k+m$ is greater by one
than the number of important digits in $n+k$. Then,
$g(n+k+m)=1-g(n+k)=1-r$ which contradicts the fact that $r$ is 
the colour of $\FS^{=3}(H)$. Thus, Claim \ref{lem:deciding} is proved.

\begin{claim}\label{lem:searching}
For each $n\in H$ there exists $k\in H$ such that  $n<k$ and $g(n+k)=r$.
\end{claim}
\noindent To prove Claim \ref{lem:searching},
fix $n$  and, again,
let $\ell$ be such that for all $x<\lambda(n)$,
\[
x\in \rg(f) \iff x\in \rg(f\restriction \ell).
\]
Take any $k\in H$ such that $\lambda(k)>\ell$.
For any $m\in H$, if $k<m$, then  $g(n+k)=g(n+k+m)=r$.
This proves Claim \ref{lem:searching}.

We now describe an algorithm for deciding membership in $\rg(f)$
given access to $H$.
For an input $x$, find $n\in H$ such that $x<\lambda(n)$.
Then, find $k\in H$ such that $n<k$ and $g(n+k)=r$. By Claim~\ref{lem:searching}
this part of computation ends successfully. 
Finally, check whether $x\in \rg(f\restriction \mu(k))$. 
By Claim~\ref{lem:deciding} this is equivalent to $x\in\rg(f)$.
\end{proof}

Let us conclude this section with some remarks on the relations between the principles
$\HT^{=n}_k$ with apartness and $\HT^{=\ell}_p$ with apartness for arbitrary $n,\ell,k,p\geq 2$. 
Prima facie it is not obvious that, say, $\HT^{=3}_2$ with apartness implies $\HT^{=2}_8$ with apartness. 
Yet the proofs of our results above allow us to show that some of these principles are equivalent over 
$\RCA_0$.

\begin{proposition}\label{prop:exactsums}
For each $n\geq 3$ and $k > 1$, $\HT^{=3}_2$ with apartness is equivalent to 
$\HT^{=n}_k$ with apartness over $\RCA_0$.
\end{proposition}

\begin{proof}
The proof of Theorem \ref{thm:hteq3_to_aca0} obviously shows that, for $n\geq 3$, 
$\HT^{=n}_2$ with apartness implies $\ACA_0$ over $\RCA_0$. 
On the other hand, for each $n\geq 1$, $\RT^n_k$ implies $\HT^{=n}_k$ with apartness. 
Finally, it is known that for each $n \geq 3$ and $k\geq 2$, the principle
$\RT^n_k$ is equivalent to $\ACA_0$ over $\RCA_0$. Thus, $\ACA_0$ implies
$\HT^{=n}_k$ with apartness. This concludes the proof.
\end{proof}

We finally observe that, in some cases an implication from $\HT^{=m}_k$ to $\HT^{=n}_k$ (with $m > n$)
can be witnessed by a strong computable reduction. 

\begin{proposition}
For any $n,m\geq 2$ and $k\geq 2$, 
if $n$ divides $m$ then $\HT^{=n}_k$ is strongly computably reducible to $\HT^{=m}_k$.
\end{proposition}

\begin{proof}
Let $f:\Nat \to k$. Let $m = nd$. Let $H = \{ h_1, h_2, \dots \}$ with $h_1 < h_2 < \dots$ 
be a solution for the instance $f$ of $\HT^{=m}_k$.
Let $H^+$ consist of the sums of $d$ many consecutive terms of $H$, i.e., 
$H^+ = \{ h_1 + \dots + h_d, h_{d+1}+\dots + h_{2d+1},\dots\}$. Then $\FS^{=n}(H^+)$ is monochromatic.
\end{proof}

\section{Bounded Hindman and Polarized Ramsey}\label{sec:ipt}

We here consider the principle $\HT^{\leq 2}$ from Question 12 of \cite{Hin-Lea-Str:03}
from the point of view of strong computable reductions. 
Before our Theorem \ref{thm:ht24_to_aca0} the only known lower bounds on 
$\HT^{\leq 2}_k$ principles were those of Dzhafarov et al. \cite{DJSW:16} 
showing that $\HT^{\leq 2}_2$ is not provable in the base theory $\RCA_0$ and that 
the Stable Ramsey's Theorem for pairs $\SRT^2_2$ follows from $\HT^{\leq 2}_2 + \mathsf{B}\Sigma^0_2$. 
$\SRT^2_2$ is just Ramsey's Theorem for $2$-colourings of $[\Nat]^2$ restricted to 
colourings -- called {\em stable colourings} -- that eventually stabilize with respect to the second coordinate.

In this section we uncover a tight connection between $\HT^{\leq 2}$ and the 
Increasing Polarized Ramsey's Theorem for pairs $\IPT^2_2$ introduced by 
Dzhafarov and Hirst in \cite{Dza-Hir:11}, which is known to be strictly stronger than $\SRT^2_2$
(Corollary 4.12 of \cite{Pat:rain}). We show that $\IPT^2_2$ is strongly computably reducible to $\HT^{\leq 2}_4$.
As a sheer implication, this is weaker than the one from $\HT^{\leq 2}_4$ to $\RT^3_2$
in our Theorem \ref{thm:ht24_to_aca0}. However we do not know whether the 
latter can be witnessed by a strong computable reduction.

We start by recalling the definition of the Increasing Polarized Ramsey's Theorem. Let $\Nat_0$ denote $\Nat\cup\{0\}$.
\begin{definition}[Increasing Polarized Ramsey's Theorem]\label{defi:ipt} For a pair of positive integers $n$ and $k$, $\IPT^n_k$
is the following principle. 
\begin{quote}
Whenever $[\Nat_0]^n$ is $k$-coloured then there exists a sequence $(H_1,\dots,H_n)$ of infinite subsets of $\Nat$ 
such that all edges of the form $\{x_1,\dots,x_n\}$ with $x_1<\dots < x_n$, $x_i\in H_i$ have the same colour. 
\end{quote}
\end{definition}
A sequence of sets $H_1,\dots,H_n$ satisfying the above homogeneity property is referred to as an \emph{increasing  p-homogeneous sequence}.
$\IPT^2_2$ can be read as the following restriction of $\RT^2_2$: given a 2-colouring of the complete graph on $\Nat$, 
we look for an infinite bipartite graph whose forward edges all have the same colour.
It is not known whether $\IPT^2_2$ is strictly weaker than $\RT^2_2$.

We first show that $\IPT^2_2$ reduces in the sense of $\leq_{\mathrm{sc}}$ to $\HT^{=2}_2$ with apartness.
This should be contrasted with the fact that no lower bounds on $\HT^{=2}_2$ {\em without} apartness are known.

\begin{theorem}\label{thm:hteq22_to_ipt22}
For any $t\geq 2$, $\IPT^2_2$ is strongly computably reducible to $\HT^{= 2}_2$ with apartness.
\end{theorem}

\begin{proof}
Let $c:[\Nat_0]^2\to 2$ be given. Define $f:\Nat\to 2$ as follows.
\[
f(n):=
\begin{cases}
0 & \mbox{ if  } n = 2^m \mbox{ for some } m,\\
c(\lambda(n),\mu(n)) & \mbox{ otherwise.}\\
\end{cases}
\]
Note that $f$ is well-defined since $\lambda(n)<\mu(n)$ if $n$ is not of the form $2^m$. 
Let $H=\{h_1 < h_2 < \dots\}\subseteq\Nat$ witness $\HT^{= 2}_2$ with apartness for $f$.
Note that (by the apartness condition) we can assume without loss of generality that $0<\lambda(h_1)$.
Let 
\[H_1 := \{\lambda(h_{2i-1})\,:\, i \in \Nat\},
\; H_2 := \{\mu(h_{2i})\,:\, i \in \Nat\}.\] 
We claim that $(H_1,H_2)$ is a solution to $\IPT^2_2$ for $c$. 

First observe that we have
\[H_1 = \{ \lambda(h_1), \lambda(h_3), \lambda(h_5), \dots\}, \; H_2 = \{ \mu(h_2), \mu(h_4), \mu(h_6),  \dots\},\] 
with $\lambda(h_1) < \lambda(h_3) < \lambda(h_5) < \dots$  and $\mu(h_2) < \mu(h_4) < \mu(h_6) < \dots.$
This is so because $\lambda(h_1)\leq\mu(h_1)<\lambda(h_2)\leq\mu(h_2)<\dots$ by the apartness condition. 
Let the colour of $\FS^{= 2}(H)$ under $f$ be $k < 2$. We claim that $c(x_1,x_2) = k$ for every increasing pair $(x_1,x_2)\in H_1\times H_2$.
Note that $(x_1,x_2) = (\lambda(h_i),\mu(h_j))$ for some $i < j$ (the case $i=j$ is impossible by 
construction of $H_1$ and $H_2$). We have
\[c(x_1,x_2) = c(\lambda(h_i),\mu(h_j)) = c(\lambda(h_i+h_j),\mu(h_i+h_j)) = f(h_i+h_j) = k,\]
since $\FS^{= 2}(H)$ is monochromatic for $f$ with colour $k$.
This shows that $(H_1,H_2)$ is an increasing p-homogeneous sequence for $c$.
\end{proof}

\begin{corollary}\label{cor1}
$\IPT^2_2$ is strongly computably reducible to $\FUT^{\leq 2}_2$ and to $\HT^{\leq 2}_4$.
\end{corollary}

\begin{proof}
Note that the relation $\leq_{\mathrm{sc}}$ is transitive.
That $\IPT^2_2\leq_{\mathrm{sc}}\FUT^{\leq 2}_2$ follows from Theorem \ref{thm:hteq22_to_ipt22} and 
Proposition \ref{prop:ht_eq_fut}. The fact that $\IPT^2_2\leq_{\mathrm{sc}}\HT^{\leq 2}_4$ follows
from Theorem \ref{thm:hteq22_to_ipt22} and Lemma~\ref{lem:leqn_to_apart}.
\end{proof}

A proof of $\IPT^2_2\leq_{\mathrm{sc}}\HT^{\leq 2}_5$ was originally given by the first author (see~\cite{Car:16:loblo})
using a different argument.

\section{Other restrictions of Hindman's Theorem}\label{sec:misc}

In this section we present results on some restrictions of Hindman's Theorem of a different flavour. 
These restrictions are not obtained by merely bounding the number of terms of the sums for which
monochromaticity is guaranteed. Instead, it is required that all sums whose length belongs to some
structured set of integers have the same colour. Nevertheless, 
some bounds on their strength can be obtained by adapting 
the previous arguments.

\subsection{Weak Yet Strong Principles}

The first author investigated in \cite{Car:16:wys}
a family of restrictions of $\HT$ that admit proofs from Ramsey's Theorem 
yet realize the Blass-Hirst-Simpson lower bound, i.e., they are equivalent to $\ACA_0$. 
Our results from the previous sections (Theorem \ref{thm:hteq3_to_aca0}
and Proposition \ref{prop:exactsums}) 
show that the principles $\HT^{=n}_k$ with apartness are a ``weak yet strong''
family in this sense. One might read this ``weak yet strong'' phenomenon as a warning not to 
over-interpret the lower bounds for $\HT^{\leq 2}$ obtained in the previous sections. 
The simplest instance of the ``weak yet strong'' phenomenon treated in~\cite{Car:16:wys} 
is the following Hindman-Brauer Theorem (with $2$-apartness): 
\begin{quote}
Whenever $\Nat$ is 2-coloured there is an infinite 
and $2$-apart set $H\subseteq\Nat$ and {\em there exist} positive integers $a,b$ such that 
$\FS^{\{a,b,a+b,a+2b\}}(H)$ is monochromatic. 
\end{quote}
We complement the results from \cite{Car:16:wys} by showing that some apparently weaker restrictions 
of Hindman's Theorem share the same properties of the Hindman-Brauer's Theorem. 

\begin{definition}
$\HT^{\exists\{a < b\}}_2$ is the following principle: 
Whenever $\Nat$ is $2$-coloured there exists an infinite set $H\subseteq\Nat$ and positive integers 
$a < b$ such that $\FS^{\{a, b\}}(H)$ is monochromatic. 
\end{definition}

\begin{theorem}\label{thm:htab_eq_aca0}
$\HT^{\exists\{a<b\}}_2$ with apartness is equivalent to $\ACA_0$ over $\RCA_0$.
\end{theorem}

\begin{proof}
We first prove the upper bound. 
Given $f:\Nat\to 2$ let $c:[\Nat]^3\to 8$ be defined as follows: 
\[ c(x_1,x_2,x_3) := \langle f(x_1),f(x_1+x_2),f(x_1+x_2+x_3)\rangle.\]
Fix an infinite and apart set $H_0\subseteq\Nat$. 
By $\RT^3_{8}$ applied to colourings of triples from $H_0$ 
we get an infinite (and 2-apart) set $H\subseteq H_0$ monochromatic for $c$. 
Let the colour of $[H]^3$ be $(c_1,c_2,c_3)$, a binary sequence of length $3$. 
Then, for each $i\in \{1,2,3\}$, $f$ restricted to $\FS^{=i}(H)$ is a constant function
with value $c_i$. Obviously for some $3 \geq b > a > 0$ it must be that $c_a = c_b$.
Then $\FS^{ \{a,b\}}(H)$ is monochromatic under $f$.

The lower bound is proved by a minor adaptation of the proof of Proposition~\ref{prop:apht22_to_aca0}. 
As the $n$ in that proof take an $a$-term sum. Then take a $(b-a)$-term
sum as the $m$.
\end{proof}

Note that the upper bound part of the previous theorem establishes that
$\HT^{\exists\{a<b\}}_2$ with apartness is strongly computably reducible to $\RT^3_8$.
The same proof yields that the following Hindman-Schur Theorem
with apartness from \cite{Car:16:wys} implies $\RT^3_2$:
\begin{quote}
Whenever $\Nat$ is 2-coloured there is an infinite and apart
set $H\subseteq\Nat$ and {\em there exist} positive integers $a,b$ such 
that $\FS^{\{a,b,a+b\}}(H)$ is monochromatic. 
\end{quote}
Indeed, the latter principle implies $\HT^{\exists\{a<b\}}$ with apartness. Provability from $\RT^3_2$
is shown in \cite{Car:16:wys} by an argument similar to the upper bound part of Theorem~\ref{thm:htab_eq_aca0}.
The proof shows indeed that the Hindman-Schur Theorem with apartness is strongly computably reducible
to $\RT^6_{2^6}$. The number $6$ comes from the Ramsey number for ensuring a monochromatic triangle
and from the standard proof of Schur's Theorem from the finite Ramsey Theorem (see, e.g., \cite{GRS}).

Let us observe that the proof of Theorem \ref{thm:hteq3_to_aca0}
works in the case of $\HT^{=a}_2$ with
apartness, for any fixed $a\geq 3$ by taking a sum 
of $a-2$ elements in place of $n$.
This leads us to the following definition and corollary.

\begin{definition}
Let $\HT^{\exists\{a\geq 3\}}_2$ be the following principle:
For every colouring $f:\Nat\to 2$ there exists an infinite set $H\subseteq\Nat$
and there exists a number $a\geq 3$ such that $\FS^{\{a\}}(H)$ is monochromatic for $f$.
\end{definition}

\begin{theorem}\label{thm:hta_to_aca0}
$\HT^{\exists \{a\geq 3\}}_2$ with apartness is equivalent to $\ACA_0$, over $\RCA_0$. 
\end{theorem}

Note that the latter result, coupled with the results of the previous section, shows that the 
principles $\HT^{=n}_k$ with apartness form a weak yet strong family in the sense of 
\cite{Car:16:wys}. 

\subsection{Increasing Polarized Hindman's Theorem}

We define an (increasing) polarized version of Hindman's Theorem. We prove that
the case of pairs and 2 colours with an appropriately defined notion of 
apartness is equivalent to $\IPT^2_2$. One of the directions is witnessed by a strong computable 
reduction.

\begin{definition}[(Increasing) Polarized Hindman's Theorem]
Fix $n\geq 1$. $\PHT_2^n$ (resp.~$\IPHT_2^n$) 
is the following principle: For every $2$-colouring $f$ of the positive integers 
there exists a sequence $(H_1,\dots,H_n)$ of infinite sets
such that for some colour $k < 2$, for all (resp.~increasing) $(x_1,\dots,x_n) \in H_1\times \dots \times H_n$, 
$f(x_1+\dots + x_n)=k$. 
\end{definition}

We impose an apartness condition on a solution $(H_1,\dots,H_n)$ of $\IPHT^n_2$
by requiring that the union $H_1\cup\dots\cup H_n$ is apart. We denote
by ``$\IPHT^n_2$ with apartness'' the principle $\IPHT^n_2$ with this apartness 
condition on the solution set. 

\begin{theorem}\label{thm:ipht22_eq_ipt22}
$\IPT^2_2$ and $\IPHT^2_2$ with apartness are equivalent over $\RCA_0$. Furthermore, 
$\IPT^2_2 \leq_{\mathrm{sc}} \IPHT^2_2$.
\end{theorem}

\begin{proof}
We first prove that $\IPT^2_2$ implies $\IPHT^2_2$ with apartness. 
Given $f:\Nat\to 2$ define $c:~[\Nat]^2\to 2$ in the obvious way setting $c(x,y) := f(x+y)$. 
Fix two infinite disjoint sets $S_1,S_2 \subseteq\Nat$ such that $S_1\cup S_2$ is apart. By Lemma 4.3
of \cite{Dza-Hir:11}, $\IPT^2_2$ implies over $\RCA_0$ its own relativization: there exists
an increasing p-homogeneous sequence $(H_1,H_2)$ for $c$ such that $H_i\subseteq S_i$.
Note that it is unclear whether this implication can be witnessed by a strong computable
reduction. The set $H_1\cup H_2$ is 2-apart by construction. Let the
colour be $k<2$. Obviously we have that for any increasing
pair $(x_1,x_2) \in H_1\times H_2$, $f(x_1+x_2) = c(x_1,x_2) = k$. 
Therefore $(H_1,H_2)$ is a solution to $\IPHT^2_2$ with apartness for $f$.

Next we prove that $\IPHT^2_2$ with apartness implies $\IPT^2_2$ and, indeed, that
$\IPT^2_2 \leq_{\mathrm{sc}} \IPHT^2_2$ with apartness.
Let $c:[\Nat_0]^2\to 2$ be given. Define $f:\Nat\to 2$ by setting $f(n) := c(\lambda(n),\mu(n))$
if $n$ is not a power of $2$ and $f(n) = 0$ otherwise. 
Let $(H_1,H_2)$ be an apart solution to $\IPHT^2_2$ for $f$, 
of colour $k<2$. Let 
$H = \{ h_1 <  h_2 < h_3 < \dots \}$ be such that 
$h_{2i-1} \in H_1$ and $h_{2i} \in H_2$ for each $i \in \Nat$.
Then set $H_1^+ := \{ \lambda(h_{2i-1}) : i \in \Nat\}$ and 
$H_2^+ := \{\mu(h_{2i}) : i \in \Nat\}$.  
We claim that $(H^+_1,H^+_2)$ is an increasing p-homogeneous pair for $c$. Let $(x_1,x_2)\in H_1^+ \times H_2^+$ be an increasing pair. 
Then for some $h\in H_1$ and $h'\in H_2$ such that $h<h'$ we have $\lambda(h)=x_1$ and $\mu(h')=x_2$.
Therefore 
\[c(x_1,x_2) = c(\lambda(h),\mu(h')) = c(\lambda(h+h'),\mu(h+h')) = f(h+h') = k,\]
regardless of the choice of $(x_1,x_2)$.
\end{proof}


\subsection{Exactly Large Sums, with apartness}

By analogy with the Pudl\'ak-R\"odl \cite{Pud-Rod:82} theorem $\RT^{!\omega}_2$ on colourings of exactly large sets 
we consider a restriction of Hindman's Theorem to exactly large sums, i.e., sums
whose set of terms is an exactly large set. 
As noted earlier, the Pudl\'ak-R\"odl theorem is known to imply $\HT$ over $\RCA_0$ 
(yet no combinatorial proof is known).  

Let us introduce some terminology and notation and state the Pudl\'ak-R\"odl theorem. 
A finite set $S\subseteq\Nat$ is {\em exactly large}, or $!\omega$-large, if $|S|=\min(S)+1$. 
Exactly large sets are strictly related to Schreier sets in Banach Space Theory (see \cite{Far-Neg:08}), 
while their supersets -- called {\em relatively large sets} -- 
play a prominent role in the study of 
unprovability results for first-order theories of arithmetic (see \cite{Par-Har:77,Ket-Sol:81}). 

\begin{definition}[Ramsey's Theorem for exactly large sets]\label{defi:rtomega}
$\RT^{!\omega}_2$ is the following principle:
\begin{quote}
Whenever the exactly large subsets of an infinite set $X$ of natural numbers are coloured in $2$ colours, 
there exists an infinite set $H\subseteq X$ such that all exactly large subsets of $H$ have the same colour. 
\end{quote}
\end{definition}
The strength of $\RT^{!\omega}_2$ was studied by the first and fourth author in \cite{Car-Zda:14} 
and proved there to be much beyond the strength of Ramsey's Theorem. 

We now formulate our analogue for Hindman's Theorem. 
Given a set $X$ of natural numbers, the 
sums of integers whose underlying set of terms is an exactly large set in $X$ are called {\em exactly large sums} (from $X$).
We denote by $\FS^{!\omega}(X)$ the set of numbers that can be expressed as sums of an exactly large
subset of $X$.

\begin{definition}[Hindman's Theorem for Exactly Large Sums]
$\HT^{!\omega}_2$ denotes the following principle: For every colouring $f:\Nat\to 2$
there exists an infinite set $H\subseteq \Nat$ such that $\FS^{!\omega}(H)$ is monochromatic
under $f$.
\end{definition}
Besides being a restriction of $\HT$, 
$\HT^{!\omega}_2$ (with $t$-apartness, for any $t>1$) has an easy direct proof from $\RT^{!\omega}_2$.
Given $f:\Nat\to 2$ just set $c(S):=f(\sum S)$, for $S$ an exactly large set (to get $t$-apartness, 
restrict $c$ to an infinite $t$-apart set). Consistently with the previous conventions, we use $\HT^{!\omega}_2$ with $2$-apartness 
to denote the principle obtained from $\HT^{!\omega}_2$ by imposing that the solution is a $2$-apart set. 
We note, however, that for the principle $\HT^{!\omega}_2$ the choice of $t$ 
in the $t$-apartness conditon might matter. 

The argument of Theorem~\ref{thm:hteq3_to_aca0} can be easily adapted to show that $\HT^{!\omega}_2$ with 
$2$-apartness implies $\ACA_0$. In the proof of Theorem~\ref{thm:hteq3_to_aca0} take, 
instead of $n$, an almost exactly large
sum $n_0+n_1+\dots+n_{n_0-2}$ of elements of $H$. The argument then proceeds unchanged.

\begin{proposition}\label{prop:htomega_to_aca0}
$\HT^{!\omega}_2$ with apartness implies $\ACA_0$ over $\RCA_0$.
\end{proposition}

Furthermore, a number of strong computable reductions can be established for 
Hindman's Theorem for exactly large sums. For example, we have the following result.

\begin{proposition}\label{prop:htomega_to_ipt22}
$\IPHT^2_2$ with apartness is strongly computably reducible to $\HT^{!\omega}_2$ 
with apartness.
\end{proposition}
\begin{proof}
Let $f:\Nat\to 2$ be given, and let $H = \{ h_1, h_2, h_3, \dots\}$ with $h_1 < h_2 < h_3 < \dots$ be an infinite 2-apart set such that $\FS^{!\omega}(H)$ is monochromatic for $f$ of colour $k < 2$. 
Let $S_1,S_2,S_3,\dots$ be such that each $S_i$ is an exactly large subset of $H$, 
$\bigcup_{i\in\Nat}S_i = H$, and $\max{S_i}<\min{S_{i+1}}$, for each $i\in\Nat$. Let $s_i = \sum S_i$.
Let $H_s := \{s_1, s_2, \dots\}$. $H_s$ is 2-apart and consists of the sums of consecutive disjoint exactly large subsets
of $H$. 
Let $H_t = \{t_1, t_2, \dots\}$ (in increasing order) be the set consisting of the elements from $H_s$ minus their largest term (when written as $!\omega$-sums). Note that distinct elements of $H_s$ share no term, because $H_s$ is 2-apart. Let $H_1 := H_t$ and $H_2 := \{s_i - t_i : i \in \Nat\}$. Then $(H_1, H_2)$ is a 2-apart solution for $\IPHT^2_2$. Note that both $H_1$ and $H_2$ are
computable relative to $H$. 
\end{proof}
Other results on $\HT^{!\omega}_2$ were proved by the third author 
in his BSc.~Thesis \cite{Lep:thesis}. For instance, 
the following implications hold over the base theory $\RCA_0$:
$\HT^{!\omega}_2$ implies $\HT^{=2}_2$, $\HT^{!\omega}_2$ with apartness implies
$\forall n \HT^{=2^n}_2$, $\HT^{!\omega}_2$ implies $\forall n \PHT^n_2$.
We believe that the study of the strength of $\HT^{!\omega}_2$ is of interest.

\section{Conclusion and some open questions}\label{sec:conclusion}

Our results are summarized in Table 1, along with previously known results. In the table we use Ramsey-theoretic 
statements instead of equivalent 
theories (thus $\RT^3_2$ for $\ACA_0$ and $\forall k \RT^1_k$ instead of $\mathsf{B}\Sigma^0_2$). 

\begin{table}[t]\label{table:results}
\centering
\caption{Implications over $\RCA_0$ ($\geq,\leq$) and strong combinatorial reductions ($\geq_{\mathrm{sc}}$, $\leq_{\mathrm{sc}}$).}
\begin{tabular}{l||c|c}
\hline\noalign{\smallskip}
Principle: & Lower Bound: & Upper Bound:\\
\noalign{\smallskip}
\hline
\noalign{\smallskip}
\hline
\noalign{\smallskip}
$\HT\equiv\FUT$ & $\geq \RT^3_2$ (\cite{Bla-Hir-Sim:87}) & $ \leq \RT^{!\omega}_2$ (\cite{Bla-Hir-Sim:87,Car-Zda:14})\\
\noalign{\smallskip}
\hline
\noalign{\smallskip}
$\HT^{\leq 2}_2$ & ? & $\leq \RT^{!\omega}_2$ (\cite{Bla-Hir-Sim:87,Car-Zda:14})\\
$\HT^{\leq 2}_2+ \forall k \RT^1_k$ & $\geq \SRT^2_2$ (\cite{DJSW:16}) & $\leq \RT^{! \omega}_2$ (\cite{Bla-Hir-Sim:87,Car-Zda:14})\\
$\FUT^{\leq 2}_2\equiv\HT^{\leq 2}_2$ with apartness & $\geq \RT^3_2$ (Prop.~\ref{prop:apht22_to_aca0}) & $\leq \RT^{! \omega}_2$ (\cite{Bla-Hir-Sim:87,Car-Zda:14})\\
$\HT^{\leq 2}_4$ & $\geq \RT^3_2$ (Th.~\ref{thm:ht24_to_aca0}), $\geq_{\mathrm{sc}}\IPT^2_2$ (Cor. \ref{cor1})  & $\leq \RT^{! \omega}_2$ (\cite{Bla-Hir-Sim:87,Car-Zda:14})\\
\noalign{\smallskip}
\hline
\noalign{\smallskip}
$\HT^{\exists\{a<b\}}_2$ & ? & $\leq \RT^3_2$, $\leq_{\mathrm{sc}}\RT^3_8$ (Th.~\ref{thm:htab_eq_aca0})\\ 
$\HT^{\exists\{a<b\}}_2$ with apartness & $\geq \RT^3_2$ (Th.~\ref{thm:htab_eq_aca0}) & $\leq \RT^3_2$, $\leq_{\mathrm{sc}}\RT^3_8$ (Th.~\ref{thm:htab_eq_aca0})\\ 
$\HT^{\exists\{a\geq 3\}}_2$ & ? & $\leq \RT^3_2$  (\cite{Car:16:wys})\\ 
$\HT^{\exists\{a \geq 3\}}_2$ with apartness & $\geq \RT^3_2$ (Th.~\ref{thm:hta_to_aca0}) & $\leq \RT^3_2$, $\leq_{\mathrm{sc}} \RT^{6}_{2^6}$  (\cite{Car:16:wys})\\ 
\noalign{\smallskip}
\hline
\noalign{\smallskip}
$\FUT^{=2}_2\equiv\HT^{= 2}_2$ with apartness  & $\geq_{\mathrm{sc}} \IPT^2_2$ (Th.~\ref{thm:hteq22_to_ipt22}) & $\leq_{\mathrm{sc}}\RT^2_2$ (obvious)\\
$\HT^{=3}_2$ & ? & $\leq_{\mathrm{sc}}\RT^3_2$ (obvious)\\ 
$\FUT^{=3}_2\equiv\HT^{=3}_2$ with apartness  & $\geq \RT^3_2$ (Th.~\ref{thm:hteq3_to_aca0}) & $\leq_{\mathrm{sc}}\RT^3_2$ (obvious)\\ 
\hline
\noalign{\smallskip}
$\IPHT^2_2$ with apartness & $\geq_{\mathrm{sc}} \IPT^2_2$ (Th.~\ref{thm:ipht22_eq_ipt22}) & $\leq \IPT^2_2$ (Th.~\ref{thm:ipht22_eq_ipt22})\\
\hline
\noalign{\smallskip}
$\HT^{!\omega}_2$ & ? & $\leq_{\mathrm{sc}} \RT^{!\omega}_2$ (obvious)\\
$\HT^{!\omega}_2$ with apartness & $\geq \RT^3_2$ (Prop.~\ref{prop:htomega_to_aca0}) & $\leq_{\mathrm{sc}} \RT^{!\omega}_2$ (obvious)\\
\noalign{\smallskip}
\hline
\end{tabular}
\end{table}

Our main result, Theorem \ref{thm:ht24_to_aca0}, showing that the $\RT^3_2$ lower bound known for $\HT$
already holds for $\HT^{\leq 2}$, might be read as indicating that the latter 
restriction is as strong as the full theorem, thus pointing to a negative answer to 
Question 12 of \cite{Hin-Lea-Str:03}. On the other hand, many of our additional results confirm the ``weak yet strong'' phenomenon
uncovered in \cite{Car:16:wys}: the known lower bounds on Hindman's Theorem hold for restricted versions 
for which --- contrary to the $\HT^{\leq n}$ restrictions studied in~\cite{DJSW:16} --- a matching 
$\RT^3_2$ upper bound is known. Analogously, the $\IPT^2_2$ lower bound for $\HT^{\leq 2}$
already holds for the principle $\HT^{=2}_2$ with apartness, which is provable from $\RT^2_2$
(for another example at this level, see \cite{Car:16:weak}).
Our results also highlight the role of the apartness condition on the solution set. They also apply to bounded versions 
of the Finite Unions formulation of Hindman's Theorem, in which an analogous condition is already built-in.

Many natural questions remain, besides the main open problems on $\HT$ and $\HT^{\leq 2}$ (Question 9 of \cite{Mon:11:open}
and Question 12 of \cite{Hin-Lea-Str:03}). The question of whether some of the known implications between Ramsey-type
theorems and Hindman-type theorems can be witnessed by strong computable reductions is of interest. We expect that
many separations are within reach of currently available methods. Some separations can be gleaned from our results and known results from the literature. For example, 
$\RT^4_8, \RT^3_9 \nleq_{\mathrm{sc}} \HT^{\exists\{a,b\}}_2$ with apartness, 
and $\RT^4_2, \RT^3_4 \nleq_{\mathrm{sc}} \HT^{=3}_2$ with apartness.
To see this, note that on the one hand we have $\HT^{\exists\{a,b\}}$ with 2-apartness $\leq_{\mathrm{sc}} \RT^3_8$
by the upper bound proof in \cite{Car:16:wys}, 
and $\HT^{=3}_2$ with 2-apartness $\leq_{\mathrm{sc}} \RT^3_2$ by the trivial proof.
On the other hand, $\RT^4_k \nleq_{\mathrm{sc}} \RT^3_k$, $\RT^3_9 \nleq_{\mathrm{sc}} \RT^3_8$  and $\RT^3_4 \nleq_{\mathrm{sc}} \RT^3_2$ (see, e.g., \cite{Pat:16}).  Note that the separations can strenghtened to computable
reducibility. 

We would like to single out the following two questions which seem to be of some general combinatorial interest. 

\begin{question}
Is there a strong computable reduction of $\IPT^3_k$ to $\HT^{\leq n}_\ell$, for some $n,k,\ell\geq 2$? 
\end{question}
On the one hand we know that the implication from $\HT^{\leq 2}_4$ to $\IPT^3_2$ holds over $\RCA_0$.
This follows from Theorem~\ref{thm:ht24_to_aca0} and the equivalence of $\IPT^3_2$ with $\RT^3_2$ (see \cite{Dza-Hir:11}). 
On the other hand, we do not know how to lift 
the combinatorial reduction $\IPT^2_2 \leq_{\mathrm{sc}}\HT^{\leq 2}_4$ of Corollary~\ref{cor1}
to higher exponents.

\begin{question}
Is there a strong computable reduction of $\HT$ to $\RT^{!\omega}_2$?
\end{question}
Combining the results of~\cite{Bla-Hir-Sim:87} and~\cite{Car-Zda:14}
we know that the implication from $\RT^{!\omega}_2$ to $\HT$ holds over $\RCA_0$. Can this be witnessed by a 
strong computable reduction? More informally: is there a combinatorial proof of Hindman's Theorem from 
the Pudl\'{a}k-R\"odl Theorem?

\end{document}